\documentclass[11pt]{article}
       \usepackage{amsfonts}
       \usepackage{stmaryrd}
       \usepackage{arydshln}
       \usepackage{graphicx} 
       \usepackage{multirow}
       \usepackage{latexsym,amssymb,mathrsfs,fancyhdr}
       \font\tenmsb=msbm10
       \font\sevenmsb=msbm7
       \font\fivemsb=msbm5
       \catcode`\@=11
       \ifx\amstexloaded@\relax\catcode`\@=\active
       \endinput\else\let\amstexloaded@\relax\fi
       \def\spaces@{\space\space\space\space\space}
       \def\spaces@@{\spaces@\spaces@\spaces@\spaces@\spaces@}
       \def\space@.  {\futurelet\space@\relax}
       \space@.   %
       \def\Err@#1{\errhelp\defaulthelp@\errmessage{AmS-TeX error: #1}}
       \def\relaxnext@{\let\next\relax}
       \def\accentfam@{7}
       \def\noaccents@{\def\accentfam@{0}}
       \def\Cal{\relaxnext@\ifmmode\let\next\Cal@\else
       \def\next{\Err@{Use \string\Cal\space only in math mode}}\fi\next}
       \def\Cal@#1{{\Cal@@{#1}}}
       \def\Cal@@#1{\noaccents@\fam\tw@#1}
       \def\Bbb{\relaxnext@\ifmmode\let\next\Bbb@\else
       \def\next{\Err@{Use \string\Bbb\space only in math mode}}\fi\next}
       \def\Bbb@#1{{\Bbb@@{#1}}}
       \def\Bbb@@#1{\noaccents@\fam\msbfam#1}
       \newfam\msbfam
       \textfont\msbfam=\tenmsb
       \scriptfont\msbfam=\sevenmsb
       \scriptscriptfont\msbfam=\fivemsb
\makeatletter % `@' now normal "letter"
\@addtoreset{equation}{section}
\makeatother  % `@' is restored as "non-letter"

\usepackage{dsfont}
\usepackage[german,english]{babel}
\usepackage{amsmath,amssymb}
\usepackage[square, comma, sort&compress, numbers]{natbib}
\usepackage{cases}
\usepackage{mathrsfs}
\usepackage{amsmath}
\usepackage{amsthm}
\usepackage{amsfonts}
\usepackage{amssymb}
\usepackage{latexsym}
\usepackage{fancyhdr}
\usepackage{extarrows}
\usepackage{geometry}
\newtheorem{thm}{Theorem}[section]
\newtheorem{prop}[thm]{Proposition}
\newtheorem{lem}[thm]{Lemma}

\newtheorem{iteration lemma}[thm]{iteration Lemma}
\newtheorem{cor}[thm]{Corollary}
\newtheorem{defn}[thm]{Definition}

\newtheorem*{acknowledgements*}{ACKNOWLEDGEMENTS}

\begin{document}

\setlength{\columnsep}{5pt}

\title{\bf Representations and properties of the $W$-weighted core-EP inverse}
\author{Yuefeng Gao$^{1,2}$\footnote{ E-mail: yfgao91@163.com},
 \ Jianlong Chen$^{1}$\footnote{ Corresponding author. E-mail: jlchen@seu.edu.cn},
 \ Pedro Patr\'{i}cio$^{2}$ \footnote{ E-mail: pedro@math.uminho.pt}
 \\\\
$^{1}$School of  Mathematics, Southeast University,  Nanjing 210096,  China;~~~~~~~~~~~~~~~ \\$^{2}$CMAT-Centro de Matem\'{a}tica, Universidade do Minho, Braga 4710-057, Portugal}

\date{}

\maketitle
\begin{quote}
{\small  In this paper, we investigate the weighted core-EP inverse introduced by Ferreyra, Levis and Thome.
Several  computational representations of the weighted core-EP inverse are obtained in terms of  singular-value decomposition, full-rank decomposition and QR decomposition. These representations are expressed in terms of various matrix powers as well as matrix product involving the core-EP inverse, Moore-Penrose inverse and usual matrix inverse.  Finally, those representations involving only Moore-Penrose inverse are compared and analyzed via computational complexity and numerical examples.


}
\end{quote}

\section{ Introduction }
Let $\mathbb{C}^{m\times n}$ be the set of all $m\times n$ complex matrices and let $\mathbb{C}_r^{m\times n}$ be the set of all $m\times n$ complex matrices of rank $r$.  For each complex matrix $A\in \mathbb{C}^{m\times n}$, $A^*$,   $\mathcal{R}_s(A)$, $\mathcal{R}(A)$ and $\mathcal{N}(A)$ denote the conjugate transpose, row space,  range (column space) and null space of $A$, respectively. The index of $A\in \mathbb{C}^{n\times n}$, denoted by ind$(A)$, is the smallest non-negative integer $k$ for which rank$(A^k)=$rank$(A^{k+1})$.
The Moore-Penrose inverse (also known as the pseudoinverse) of $A\in \mathbb{C}^{m\times n}$, Drazin inverse of $A\in \mathbb{C}^{n\times n}$ are denoted as usual by $A^{\dag}$, $A^{D}$ respectively.

The Drazin inverse was extended to a rectangular matrix by Cline and Greville \cite{C1980}. 
Let $A\in \mathbb{C}^{m\times n},~W\in \mathbb{C}^{n\times m}$ and $k=$max$\{$ind$(AW),$ ind$(WA)\}$. The $W$-weighted Drazin inverse of $A$, denoted by $A^{D,W}$, is the unique solution to 
$$(AW)^k=(AW)^{k+1}XW,~X=XWAWX~\text{and}~AWX=XWA.$$
Many authors have been focusing on the weighted Drazin inverse and have achieved much in the aspect of representations (see for example, \cite{ZH2007,M2017,W2003}).

Baksalary and Trenkler \cite{B2010} introduced the notion of core inverse for a square  matrix of index one.
Then, 
Manjunatha Prasad and Mohana \cite{2014D} proposed the core-EP inverse for a square  matrix of arbitrary index,
as an extension of  the core inverse. Later, Gao and Chen \cite{G2018} gave a characterization for the core-EP inverse in terms of three equations. The core-EP inverse of $A\in \mathbb{C}^{n\times n}$, denoted by $A^{\scriptsize{\textcircled{\tiny \dag}}}$, is the unique solution to 
\begin{equation}
XA^{k+1}=A^k,~AX^2=X~\text{and}~(AX)^*=AX, 
\end{equation}
where $k=$ind$(A)$. 
The core-EP inverse is an
outer inverse (resp. \{2\}-inverse), i.e., $A^{\scriptsize{\textcircled{\tiny \dag}}}AA^{\scriptsize{\textcircled{\tiny \dag}}}=A^{\scriptsize{\textcircled{\tiny \dag}}}$.  The core-EP inverse has the following properties: 

(1)~$\mathcal{R}(A^{{\scriptsize{\textcircled{\tiny \dag}}}})=\mathcal{R}(A^k)$, $\mathcal{N}(A^{{\scriptsize{\textcircled{\tiny \dag}}}})=\mathcal{N}((A^k)^*),$

(2)~$\mathcal{R}(A^{{\scriptsize{\textcircled{\tiny \dag}}}})\oplus \mathcal{N}(A^{{\scriptsize{\textcircled{\tiny \dag}}}})=\mathbb{C}^{n\times n}$,

(3)~$AA^{{\scriptsize{\textcircled{\tiny \dag}}}}$ is an orthogonal projector onto $\mathcal{R}(A^k)$ and $A^{{\scriptsize{\textcircled{\tiny \dag}}}}A$ is an oblique projector on to 

$\mathcal{R}(A^k)$ along $\mathcal{N}((A^k)^{\dag}A)$.
 \\
The core inverse and core-EP inverse have applications in partial order theory (see for example, \cite{G2017,2017,W2016}).

Recently, an extension of the core-EP inverse from a square matrix to a rectangular matrix was made by 
Ferreyra et al. \cite{F2017}. Let $A\in \mathbb{C}^{m\times n},~W\in \mathbb{C}^{n\times m}$ and $k=$max$\{$ind$(AW),$ ind$(WA)\}$. The $W$-weighted core-EP inverse of $A$, denoted by $A^{\scriptsize{\textcircled{\tiny \dag}},W}$,  is the unique solution to the system 
\begin{equation}
WAWX=(WA)^k[(WA)^k]^{\dag}~\text{and}~\mathcal{R}(X)\subseteq \mathcal{R}((AW)^k).
\end{equation}
Meanwhile, the authors proved that the $W$-weighted core-EP inverse of $A$ can be written as a product of matrix powers involving two Moore-Penrose inverses:
\begin{equation}
A^{\scriptsize{\textcircled{\tiny \dag}},W}=[W(AW)^{l+1}[(AW)^l]^{\dag}]^{\dag}~(l\geq k).
\end{equation}
Then, Mosi\'{c} \cite{MM2017}  studied the weighted core-EP inverse of an operator between two Hilbert spaces as a generalization of the weighted core-EP inverse of a rectangular matrix.

In this paper, our main goal is to further study the weighted core-EP inverse for a rectangular  matrix and compile its new, computable representations. The paper is carried out as follows. 
In Section 2, first of all, the weighted core-EP inverse  is characterized in terms of three equations. This could be very useful in testing the accuracy of a given numerical method (to compute the weighted core-EP inverse) via residual norms. 
Then, we derive the canonical form for the $W$-weighted core-EP inverse of $A$ by using the singular value decompositions of $A$ and $W$.
Later, representations of the weighted core-EP inverse are obtained via  full-rank decomposition, general algebraic structure (GAS) and QR decomposition in conjunction with the fact that the weighted core-EP inverse is a particular outer inverse. These representations are expressed eventually through various matrix powers as well as matrix product involving the core-EP inverse, Moore-Penrose inverse and usual matrix inverse. 
In Section 3, some properties of the weighted core-EP inverse are exhibited naturally as outcomes of given representations. 
As mentioned earlier, the weighted core-EP inverse is a particular outer inverse. It is known that  the inverse along an element \cite{2011} and $(B,C)$-inverse \cite{D2012} are outer inverses as well. Thus, 
in Section 4,
we wish to reveal the  relations among the weighted core-EP inverse, weighted Drazin inverse, the inverse along an element, and $(B,C)$-inverse.   
In Section 5, the computational complexities of  proposed representations involving pseudoinverse are estimated. 
In the last Section 6,  corresponding numerical examples  are implemented by using Matlab R2017b. 
   
\section{Representations of the weighted core-EP inverse}
In this section, we compile some new expressions  of the weighted core-EP inverse for a rectangular complex matrix. First, the weighted core-EP inverse is characterized in terms of three equations.This plays a key role in  examining the accuracy of a numerical method.
\begin{lem}\label{x}\emph{\cite[Theorem 2.3]{G2018}}~Let $A\in \mathbb{C}^{n\times n}$ and  let $l$ be a non-negative integer such that $l\geq k={\rm ind}(A)$.  Then  $A^{\scriptsize{\textcircled{\tiny \dag}}}=A^DA^{l}(A^{l})^{\dag}$. In this case, $AA^{\scriptsize{\textcircled{\tiny \dag}}}=A^{l}(A^{l})^{\dag}$.
\end{lem}

\begin{thm}\label{2.1} Let $A\in \mathbb{C}^{m\times n},~W\in \mathbb{C}^{n\times m}$ and $k=\mathrm{max}\{{\rm ind}(AW),{\rm ind}(WA)\}$. Then  there exists a unique $X\in \mathbb{C}^{m\times n}$ such that 
\begin{equation}
XW(AW)^{k+1}=(AW)^k,~AWXWX=X~\text{and}~(WAWX)^*=WAWX.
\end{equation}
The unique $X$ which satisfies the above equations is $X=A[(WA)^{\scriptsize{\textcircled{\tiny \dag}}}]^2$.
\end{thm}

\begin{proof}~First of all, we can check that $X=A[(WA)^{\scriptsize{\textcircled{\tiny \dag}}}]^2$ satisfies the equations in (2.1). In fact, in view of Lemma \ref{x},
\begin{equation*}
\begin{aligned}
A[(WA)^{\scriptsize{\textcircled{\tiny \dag}}}]^2W(AW)^{k+1}
&=A(WA)^{\scriptsize{\textcircled{\tiny \dag}}}[(WA)^{\scriptsize{\textcircled{\tiny \dag}}}(WA)^{k+1}]W=A(WA)^{\scriptsize{\textcircled{\tiny \dag}}}(WA)^{k}W\\
&=A(WA)^{D}(WA)^{k}[(WA)^{k}]^{\dag}(WA)^{k}W\\
&=A(WA)^{D}(WA)^{k}W,~\text{which~implies~that}~\\
A[(WA)^{\scriptsize{\textcircled{\tiny \dag}}}]^2W(AW)^{k+1}
&=(AW)^{D}(AW)^{k+1}=(AW)^{k},~\text{since~}A(WA)^{D}=(AW)^{D}A;\\
AWA[(WA)^{\scriptsize{\textcircled{\tiny \dag}}}]^2WA[(W&A)^{\scriptsize{\textcircled{\tiny \dag}}}]^2
=A(WA)^{\scriptsize{\textcircled{\tiny \dag}}}WA[(WA)^{\scriptsize{\textcircled{\tiny \dag}}}]^2=A[(WA)^{\scriptsize{\textcircled{\tiny \dag}}}]^2;\\
 (WAWA[(WA)^{\scriptsize{\textcircled{\tiny \dag}}}]^2)^*=&WA(WA)^{\scriptsize{\textcircled{\tiny \dag}}}.
 \end{aligned}
\end{equation*}
Next, we would give a proof of the uniqueness of $X$. If $$XW(AW)^{k+1}=(AW)^k,~AWXWX=X~\text{and}~(WAWX)^*=WAWX$$ and
$$YW(AW)^{k+1}=(AW)^k,~AWYWY=Y~\text{and}~(WAWY)^*=WAWY,$$
then 
\begin{equation*}
\begin{aligned}
X&=AWXWX=(AW)^2(XW)^2X=(AW)^k(XW)^kX\\
&=YW(AW)^{k+1}(XW)^{k}X=Y(WA)^{k+1}(WX)^{k+1}=Y(WA)^{k+2}(WX)^{k+2}\\
&=Y[(WA)^{k+2}(WY)^{k+2}(WA)^{k+2}](WX)^{k+2}\\
&=Y[(WA)^{k+2}(WY)^{k+2}]^*[(WA)^{k+2}(WX)^{k+2}]^*\\
&=Y[(WY)^{k+2}]^*[(WA)^{k+2}(WX)^{k+2}(WA)^{k+2}]^*=Y[(WY)^{k+2}]^*[(WA)^{k+2}]^*\\
&=Y(WAWY)^*=YWAWY=Y(WA)^{k+1}(WY)^{k+1}=YW(AW)^{k+1}(YW)^kY\\
&=(AW)^k(YW)^kY=AWYWY=Y.
\end{aligned}
\end{equation*}
This completes the proof.
\end{proof}

\begin{thm}\label{2.2} Let $A,~X\in \mathbb{C}^{m\times n},~W\in \mathbb{C}^{n\times m}$ and $k=\mathrm{max}\{{\rm ind}(AW),{\rm ind}(WA)\}$. Then  the following are equivalent:\\
$(1)$~$A^{{\scriptsize{\textcircled{\tiny \dag}}}, W}=X;$\\
$(2)$~$XW(AW)^{k+1}=(AW)^k,~AWXWX=X~\text{and}~(WAWX)^*=WAWX.$ 
\end{thm}

\begin{proof} It suffices to show that $X=A[(WA)^{\scriptsize{\textcircled{\tiny \dag}}}]^2$ satisfies condition (1.2).
Indeed, 
\begin{equation*}
\begin{aligned}
&WAWA[(WA)^{\scriptsize{\textcircled{\tiny \dag}}}]^2=WA(WA)^{\scriptsize{\textcircled{\tiny \dag}}}=WA(WA)^D(WA)^k[(WA)^k]^{\dag}=(WA)^k[(WA)^k]^{\dag},
\\ 
&A[(WA)^{\scriptsize{\textcircled{\tiny \dag}}}]^2=AWA[(WA)^{\scriptsize{\textcircled{\tiny \dag}}}]^3
=A(WA)^k[(WA)^{\scriptsize{\textcircled{\tiny \dag}}}]^{k+2}=(AW)^kA[(WA)^{\scriptsize{\textcircled{\tiny \dag}}}]^{k+2}, \text{i.e.},\\
&\mathcal{R}(A[(WA)^{\scriptsize{\textcircled{\tiny \dag}}}]^2) \subseteq \mathcal{R}((AW)^k).
\end{aligned}
\end{equation*} 
This completes the proof.
\end{proof}

%\begin{cor} Let $A\in \mathbb{C}^{m\times n},~W\in \mathbb{C}^{n\times m}$. Then $A^{{\scriptsize{\textcircled{\tiny \dag}}}, W}=A$ if and only if $A=A(WAWA)^*$.
%\end{cor}

\vspace{4mm}

We now give the canonical form for the $W$-weighted core-EP inverse of $A$ by using the singular value decompositions of $A$ and $W$. Let $A\in \mathbb{C}_r^{m\times n},~W\in \mathbb{C}_s^{n\times m}$ be of the following singular value decompositions:
\begin{equation}
A=U\begin{pmatrix}
       \Sigma_{1} & 0\\
       0 &0
      \end{pmatrix}V^*~\text{and}~W=S\begin{pmatrix}
       \Sigma_{2} & 0\\
       0 &0
      \end{pmatrix}T^*,
\end{equation}
where $U,~V,~S,~T$ are unitary matrices, $\Sigma_{1}=$diag$(\sigma_1,\cdots,\sigma_r),~\sigma_1\geq \cdots \geq \sigma_r>0$, entries  $\sigma_i$  are known as the singular values of 
$A$,  and $\Sigma_{2}=$diag$(\tau_1,\cdots,\tau_s),~\tau_1\geq \cdots \geq \tau_s>0$, entries $\tau_i$ are   singular values of 
$W$. 

\begin{thm}~Let $A\in \mathbb{C}_r^{m\times n},~W\in \mathbb{C}_s^{n\times m}$ be of the singular value decompositions as in $(2.2)$. Then 
\begin{equation}
A^{{\scriptsize{\textcircled{\tiny \dag}}}, W}=U\begin{bmatrix}
                                                         \Sigma_{1}H_1[(\Sigma_{2}R_1\Sigma_{1}H_1)^{\scriptsize{\textcircled{\tiny \dag}}}]^2&0\\
                                                                       0&0
                                                                       \end{bmatrix}S^*,
\end{equation}
where $T^*U=\begin{bmatrix}
                              R_1&R_2\\
                               R_3&R_4
                       \end{bmatrix},~R_1\in \mathbb{C}^{s\times r},~
                       V^*S=\begin{bmatrix}
                              H_1 & H_2\\
                               H_3 & H_4
                       \end{bmatrix},~H_1\in \mathbb{C}^{r\times s}$.
\end{thm}
\begin{proof}~
Observe that $WA=S\begin{bmatrix}
                        \Sigma_{2}&0\\
                        0&0
                    \end{bmatrix}T^*U\begin{bmatrix}
                        \Sigma_{1}&0\\
                        0&0
                    \end{bmatrix}V^*=S\begin{bmatrix}
                        \Sigma_{2}R_1 \Sigma_{1}&0\\
                        0&0
                    \end{bmatrix}V^*$. Now suppose that ind$(WA)=k$ and $X=S\begin{bmatrix}
                              X_1 & X_2\\
                               X_3 & X_4
                       \end{bmatrix}S^*~(X_1\in \mathbb{C}^{s\times s})$ is the core-EP inverse of $WA$, then $X$ would satisfy condition (1.1).                      
Thus, by computation,
\begin{equation*}
\begin{aligned}
&(\Sigma_{2}R_1\Sigma_{1}H_1X_1)^*=\Sigma_{2}R_1\Sigma_{1}H_1X_1,~\Sigma_{2}R_1\Sigma_{1}H_1X_2=0,\\
&\Sigma_{2}R_1\Sigma_{1}H_1X_1^2=X_1,%~\Sigma_{2}r_1\Sigma_{1}h_1x_1x_2=x_2,
~X_3=X_4=0,\\
&X_1\Sigma_{2}R_1\Sigma_{1}(H_1\Sigma_{2}R_1\Sigma_{1})^k=\Sigma_{2}R_1\Sigma_{1}(H_1\Sigma_{2}R_1\Sigma_{1})^{k-1},~\text{which~implies~that~}\\
&X_1(\Sigma_{2}R_1\Sigma_{1}H_1)^{k+1}=(\Sigma_{2}R_1\Sigma_{1}H_1)^{k}.
\end{aligned}
\end{equation*}
These equalities above show that $X_1=(\Sigma_{2}R_1\Sigma_{1}H_1)^{\scriptsize{\textcircled{\tiny \dag}}}$.   As the core-EP inverse is an outer inverse, i.e., $XWAX=X$, then $X_2=X_1\Sigma_{2}R_1\Sigma_{1}H_1X_2=0$. Hence, $$(WA)^{\scriptsize{\textcircled{\tiny \dag}}}=S\begin{bmatrix}
                              (\Sigma_{2}R_1\Sigma_{1}H_1)^{\scriptsize{\textcircled{\tiny \dag}}} & 0\\
                               0 & 0
                       \end{bmatrix}S^*.$$
In light of Theorems \ref{2.1} and \ref{2.2},  $A^{{\scriptsize{\textcircled{\tiny \dag}}}, W}=A[(WA)^{\scriptsize{\textcircled{\tiny \dag}}}]^2=U\begin{bmatrix}
                                                         \Sigma_{1}H_1[(\Sigma_{2}R_1\Sigma_{1}H_1)^{\scriptsize{\textcircled{\tiny \dag}}}]^2&0\\
                                                                       0&0
                                                                       \end{bmatrix}S^*$. This completes the proof.
\end{proof}

\vspace{4mm}

Additional representations of the weighted core-EP inverse  can be obtained through the full-rank decomposition.
First, let us recall a concerned notion.
In 1974, Ben-Israel and Greville \cite{D1974} introduced the notion of generalized inverse with prescribed range and null space. 
Let $A\in  \mathbb{C}_r^{m\times n}$, $T$ be a subspace of $\mathbb{C}^{n}$ of dimension $s\leq r$ and let $S$ be  a subspace of $\mathbb{C}^{m}$ of dimension $m-s$. If $A$ has a $\{2\}$-inverse $X$ such that $\mathcal{R}(X) = T$ and $\mathcal{N}(X) = S$, then $X$ is unique and denoted by 
$A^{(2)}_{T,S}$. 
Further, Sheng and Chen \cite{S2007}  gave a full-rank representation of the generalized inverse $A^{(2)}_{T,S}$, which 
is based on an arbitrary full-rank decomposition of $G$, where $G$ is a matrix such that $\mathcal{R}(G) = T$ and $\mathcal{N}(G) = S$. 
\begin{lem}\label{S2007}\emph{\cite[Theorem 3.1]{S2007}}~Let $A\in \mathbb{C}_r^{m\times n}$, $T$ be a subspace of $\mathbb{C}^{n}$ of dimension $s\leq r$ and let $S$ be  a subspace of $\mathbb{C}^{m}$ of dimension $m-s$. Suppose that $G\in \mathbb{C}^{n\times m}$ satisfies $\mathcal{R}(G)=T,~\mathcal{N}(G)=S$. Let $G$ be of an arbitrary full-rank decomposition, namely $G=UV$. If $A$ has a $\{2\}$-inverse $A^{(2)}_{T,S}$, then \\
$(1)~VAU~\text{is~invertible};$\\
$(2)~A^{(2)}_{T,S}=U(VAU)^{-1}V.$
\end{lem}

The following result shows that the weighted core-EP inverse  is a generalized inverse with prescribed range and null space.
\begin{thm}\label{2.8}~Let $A\in \mathbb{C}^{m\times n},~W\in \mathbb{C}^{n\times m}$ with ${\rm ind}(WA)=k$. The W-weighted core-EP inverse of $A$ is a $\{2\}$-inverse of $WAW$ with the range $\mathcal{R}(A(WA)^k[(WA)^{k}]^{\dag})$ and the null space $\mathcal{N}(A(WA)^k[(WA)^{k}]^{\dag})$ i.e.,  
\begin{equation}
A^{{\scriptsize{\textcircled{\tiny \dag}}}, W}=(WAW)^{(2)}_{\mathcal{R}(G), \mathcal{N}(G)},
\end{equation}
where $G=A(WA)^k[(WA)^{k}]^{\dag}$.
\end{thm}
\begin{proof}~First, we check that $A[(WA)^{{\scriptsize{\textcircled{\tiny \dag}}}}]^2$ is a $\{2\}$-inverse of $WAW$. Indeed, $$A[(WA)^{{\scriptsize{\textcircled{\tiny \dag}}}}]^2WAWA[(WA)^{{\scriptsize{\textcircled{\tiny \dag}}}}]^2=A[(WA)^{{\scriptsize{\textcircled{\tiny \dag}}}}]^2WA(WA)^{{\scriptsize{\textcircled{\tiny \dag}}}}=A[(WA)^{{\scriptsize{\textcircled{\tiny \dag}}}}]^2.$$
Then, we show that $\mathcal{R}(A(WA)^k[(WA)^{k}]^{\dag})=\mathcal{R}(A[(WA)^{{\scriptsize{\textcircled{\tiny \dag}}}}]^2)$ and $\mathcal{N}(A(WA)^k[(WA)^{k}]^{\dag})=\mathcal{N}(A[(WA)^{{\scriptsize{\textcircled{\tiny \dag}}}}]^2)$.
Indeed, 
\begin{equation*}
\begin{aligned}
&A(WA)^k[(WA)^{k}]^{\dag}=A[(WA)^{{\scriptsize{\textcircled{\tiny \dag}}}}]^2(WA)^{k+2}[(WA)^{k}]^{\dag}, \\
&\text{i.e.,}~\mathcal{R}(A(WA)^k[(WA)^{k}]^{\dag})\subseteq \mathcal{R}(A[(WA)^{{\scriptsize{\textcircled{\tiny \dag}}}}]^2);\\
&A[(WA)^{{\scriptsize{\textcircled{\tiny \dag}}}}]^2=A(WA)^{k}[(WA)^{{\scriptsize{\textcircled{\tiny \dag}}}}]^{k+2}=A(WA)^{k}[(WA)^{k}]^{\dag}(WA)^{k}[(WA)^{{\scriptsize{\textcircled{\tiny \dag}}}}]^{k+2}, \\
&\text{i.e.,}~\mathcal{R}(A[(WA)^{{\scriptsize{\textcircled{\tiny \dag}}}}]^2)\subseteq \mathcal{R}(A(WA)^k[(WA)^{k}]^{\dag}).
\end{aligned}
\end{equation*}
If $X\in \mathcal{N}(A(WA)^k[(WA)^{k}]^{\dag})$, i.e., $A(WA)^k[(WA)^{k}]^{\dag}X=0$, then 
\begin{equation*}
\begin{aligned}
A[(WA)^{{\scriptsize{\textcircled{\tiny \dag}}}}]^2X
&=A(WA)^{{\scriptsize{\textcircled{\tiny \dag}}}}(WA)^D(WA)^k[(WA)^k]^{\dag}X\\
&=A(WA)^{{\scriptsize{\textcircled{\tiny \dag}}}}[(WA)^D]^2WA(WA)^k[(WA)^k]^{\dag}X=0,
\end{aligned}
\end{equation*}
namely, $\mathcal{N}(A(WA)^k[(WA)^{k}]^{\dag}) \subseteq \mathcal{N}(A[(WA)^{{\scriptsize{\textcircled{\tiny \dag}}}}]^2)$;\\
if $X\in \mathcal{N}(A[(WA)^{{\scriptsize{\textcircled{\tiny \dag}}}}]^2)$, i.e., $A[(WA)^{{\scriptsize{\textcircled{\tiny \dag}}}}]^2X=0$, then 
\begin{equation*}
\begin{aligned}
A(WA)^k[(WA)^{k}]^{\dag}X&=AWA(WA)^{{\scriptsize{\textcircled{\tiny \dag}}}}X=AWAWA[(WA)^{{\scriptsize{\textcircled{\tiny \dag}}}}]^2X=0,
\end{aligned}
\end{equation*}
namely, $\mathcal{N}(A[(WA)^{{\scriptsize{\textcircled{\tiny \dag}}}}]^2)  \subseteq \mathcal{N}(A(WA)^k[(WA)^{k}]^{\dag}) $.\\
This completes the proof.
\end{proof}

From Theorem \ref{2.8}, it is known  that the weighted core-EP inverse is a particular outer inverse. Then
by applying Lemma \ref{S2007}, we  derive  new  representations of the weighted core-EP inverse involving the usual matrix inverse.
\begin{cor}\label{2.9}~Let $A\in \mathbb{C}^{m\times n},~W\in \mathbb{C}^{n\times m}$ with ${\rm ind}(WA)=k$.
If $A(WA)^k[(WA)^k]^{\dag}=UV$ is a full-rank decomposition of $A(WA)^k[(WA)^k]^{\dag}$. Then the $W$-weighted core-EP inverse of $A$ possesses the following representation:
\begin{equation}
A^{{\scriptsize{\textcircled{\tiny \dag}}}, W}=U(VWAWU)^{-1}V.
\end{equation}
\end{cor}

\vspace{4mm}

Recall that the general algebraic structures (GAS) of $A$ and $W$ are defined as follows (see \cite{W2003}):
\begin{equation}
\begin{aligned}
&A=P\begin{bmatrix}
          A_{11}&0\\
          0&A_{22}
          \end{bmatrix}Q^{-1},~W=Q\begin{bmatrix}
                                                        W_{11}  & 0\\
                                                               0 & W_{22}
                                                   \end{bmatrix}P^{-1},
\end{aligned}
\end{equation}
where $P,~Q,~A_{11},~W_{11}$ are non-singular matrices and $A_{22},~W_{22},~A_{22}W_{22},~W_{22}A_{22}$ are nilpotent matrices.
\begin{cor}~Let $A\in \mathbb{C}^{m\times n},~W\in \mathbb{C}^{n\times m}$ with ${\rm ind}(WA)=k$ and let  $P=\begin{bmatrix}
                                 P_1&P_2
                                 \end{bmatrix},~Q=\begin{bmatrix}
                                 L_1&L_2
                                 \end{bmatrix}$, where $P_1,~P_2,~L_1,~L_2$ are appropriate blocks arising from $(2.6)$.
Then the $W$-weighted core-EP inverse of $A$ possesses the following representation:
\begin{equation}
A^{{\scriptsize{\textcircled{\tiny \dag}}}, W}=P_1(L_1^*WAWP_1)^{-1}L_1^*.
\end{equation}
\end{cor}
\begin{proof}～Suppose that 
                                 $Q^{-1}=\begin{bmatrix}
                                 Q_1\\
                                 Q_2
                                 \end{bmatrix}$. From the GAS representations (2.6), it follows that 
\begin{equation*}
\begin{aligned}
&(WA)^k=Q\begin{bmatrix}
          (W_{11}A_{11})^{k}&0\\
          0&0
          \end{bmatrix}Q^{-1}=L_1(W_{11}A_{11})^{k}Q_{1},\\ 
&[(WA)^k]^{\dag}=Q_1^*(Q_1Q_1^*)^{-1}[(W_{11}A_{11})^{k}]^{-1}(L_1^*L_1)^{-1}L_1^*,\\
&A(WA)^k=P_1A_{11}(W_{11}A_{11})^{k}Q_{1}~\text{and}\\
&A(WA)^k[(WA)^k]^{\dag}=P_1A_{11}(L_1^*L_1)^{-1}L_1^*.
\end{aligned}
\end{equation*}
Therefore, it is possible to use the full-rank decomposition $A(WA)^k[(WA)^k]^{\dag}=UV$, where
$$U=P_1A_{11}~\text{and}~V=(L_1^*L_1)^{-1}L_1^*.$$
Then by Corollary \ref{2.9}, we obtain $A^{{\scriptsize{\textcircled{\tiny \dag}}}, W}=P_1A_{11}[(L_1^*L_1)^{-1}L_1^*WAWP_1A_{11}]^{-1}(L_1^*L_1)^{-1}L_1^*=P_1(L_1^*WAWP_1)^{-1}L_1^*$. This completes the proof.
\end{proof}

\vspace{4mm}
The following representation of the weighted core-EP inverse  is based on the  
QR decomposition  defined as in \cite{M2017,S2012,W2002}.
\begin{cor}~Let $A\in \mathbb{C}^{m\times n},~W\in \mathbb{C}^{n\times m}$ with $\mathrm{ind}(WA)=k$, ${\rm rank}(WAW)=r$, ${\rm rank}[A(WA)^{k}[(WA)^k]^{\dag}]=s,~s\leq r$. Suppose that the QR decomposition of $A(WA)^{k}[(WA)^k]^{\dag}$ is of the form
$$A(WA)^{k}[(WA)^k]^{\dag}P=QR,$$
where $P$ is an $n\times n$ permutation matrix, $Q\in \mathbb{C}^{m\times m},~Q^*Q=I_m$ and $R\in \mathbb{C}_s^{m\times n}$ is an upper trapezoidal matrix. Assume that $P$ is chosen so that $Q$ and $R$ can be partitioned as 
$$Q=\begin{bmatrix}
Q_1&Q_2
\end{bmatrix},~R=\begin{bmatrix}
R_{11}&R_{12}\\
{\bf 0}&{\bf 0}
\end{bmatrix}=\begin{bmatrix}
R_{1}\\
{\bf 0}
\end{bmatrix},$$
where $Q_1$ consists of the first $s$ columns of $Q$ and $R_{11}\in \mathbb{C}^{s\times s}$ is non-singular. If $WAW$ has a $\{2\}$-inverse $(WAW)^{(2)}_{\mathcal{R}(G),\mathcal{N}(G)}=A^{{\scriptsize{\textcircled{\tiny \dag}}}, W}$, where $G=A(WA)^k[(WA)^{k}]^{\dag}$, then\\
$(1)~R_1P^*WAWQ_1~\text{is~an~invertible~matrix};$\\
$(2)~A^{{\scriptsize{\textcircled{\tiny \dag}}}, W}=Q_1(R_1P^*WAWQ_1)^{-1}R_1P^*;$\\
$(3)~A^{{\scriptsize{\textcircled{\tiny \dag}}}, W}=(WAW)^{(2)}_{\mathcal{R}(Q_1),\mathcal{N}(R_1P^*)};$\\
$(4)~A^{{\scriptsize{\textcircled{\tiny \dag}}}, W}=Q_1(Q_1^*A(WA)^k[(WA)^{k}]^{\dag}WAWQ_1)^{-1}Q_1^*A(WA)^k[(WA)^{k}]^{\dag}.$
\end{cor}

\vspace{4mm}
Various generalized inverses of complex matrices can be finally expressed in terms of the matrix product as well as matrix powers involving only Moore-Penrose inverse, so can the weighted core-EP inverse. It is crucial since in that case the operation could be implemented easily by  Matlab.
The main disadvantage of the representation (1.3)  arises from the necessity to calculate  Moore-Penrose inverses of two different matrices.
The following result derives a  representation of $A^{{\scriptsize{\textcircled{\tiny \dag}}}, W}$, which involves only one Moore-Penrose inverse.
\begin{thm}\label{2.5}~Let $A\in \mathbb{C}^{m\times n},~W\in \mathbb{C}^{n\times m}$ and  let $l$ be a non-negative  integer such that $l\geq k=\mathrm{max}\{{\rm ind}(AW),{\rm ind}(WA)\}$.  Then  $A^{{\scriptsize{\textcircled{\tiny \dag}}}, W}$ can be written as follows:
\begin{equation}
A^{{\scriptsize{\textcircled{\tiny \dag}}}, W}=(AW)^l[W(AW)^{l+1}]^{\dag};
\end{equation}
\begin{equation}
A^{{\scriptsize{\textcircled{\tiny \dag}}}, W}=A(WA)^l[(WA)^{l+2}]^{\dag}.
\end{equation}
\end{thm}

\begin{proof} From Theorems \ref{2.1} and \ref{2.2}, it follows that $A^{{\scriptsize{\textcircled{\tiny \dag}}}, W}=A[(WA)^{\scriptsize{\textcircled{\tiny \dag}}}]^2$. As $$(WA)^{\scriptsize{\textcircled{\tiny \dag}}}=(WA)^D(WA)^{l}[(WA)^{l}]^{\dag}=(WA)^D(WA)^{l+2}[(WA)^{l+2}]^{\dag}$$  
by Lemma \ref{x},  we derive that 
\begin{equation*}
\begin{aligned}
A^{{\scriptsize{\textcircled{\tiny \dag}}}, W}&=A[(WA)^{\scriptsize{\textcircled{\tiny \dag}}}]^2\\
&=A[(WA)^D(WA)^{l+2}[(WA)^{l+2}]^{\dag}]^2\\
&=A[(WA)^D]^2(WA)^{l+2}[(WA)^{l+2}]^{\dag}=A(WA)^{l}[(WA)^{l+2}]^{\dag}.
\end{aligned}
\end{equation*}
One can verify (2.9) by checking three equations  in Theorem \ref{2.2}. Here we omit the details.
\end{proof}
An expression of the core-EP inverse can be derived as a particular case $W=I$ of Theorem~\ref{2.5}.
\begin{cor}~Let $A\in \mathbb{C}^{n\times n}$ and  let $l$ be a positive integer such that $l\geq k={\rm ind}(A)$.  Then  $A^{{\scriptsize{\textcircled{\tiny \dag}}}}=A^{l}(A^{l+1})^{\dag}.$
\end{cor}

\section{Properties of the weighted core-EP inverse }
In this section, we  study the properties of the weighted core-EP inverse.

\begin{prop}~Let $A\in \mathbb{C}^{m\times n},~W\in \mathbb{C}^{n\times m}$ with $ k=\mathrm{max}\{{\rm ind}(AW),{\rm ind}(WA)\}$. Then we have the following facts: \\
$(1)~\mathcal{R}(A^{{\scriptsize{\textcircled{\tiny \dag}}}, W})=\mathcal{R}((AW)^k);$
\\
$(2)~\mathcal{N}(A^{{\scriptsize{\textcircled{\tiny \dag}}}, W})=\mathcal{N}([(WA)^k]^*).$
\end{prop}
\begin{proof}~(1)~In view of Theorems \ref{2.1} and \ref{2.2},  $A^{{\scriptsize{\textcircled{\tiny \dag}}}, W}=A[(WA)^{\scriptsize{\textcircled{\tiny \dag}}}]^2=A(WA)^k[(WA)^{\scriptsize{\textcircled{\tiny \dag}}}]^{k+2}=(AW)^kA[(WA)^{\scriptsize{\textcircled{\tiny \dag}}}]^{k+2}$, i.e., $\mathcal{R}(A^{{\scriptsize{\textcircled{\tiny \dag}}}, W})\subseteq \mathcal{R}((AW)^k)$, together with $$(AW)^k=A[(WA)^{\scriptsize{\textcircled{\tiny \dag}}}]^2(WA)^{k+2}W(AW)^D=A^{{\scriptsize{\textcircled{\tiny \dag}}}, W}(WA)^{k+1}W,$$ i.e., $\mathcal{R}((AW)^k)\subseteq  \mathcal{R}(A^{{\scriptsize{\textcircled{\tiny \dag}}}, W})$. Thus, $\mathcal{R}(A^{{\scriptsize{\textcircled{\tiny \dag}}}, W})=\mathcal{R}((AW)^k)$.
\\
(2) Suppose $Y\in \mathcal{N}(A^{{\scriptsize{\textcircled{\tiny \dag}}}, W})$, i.e., $A[(WA)^{\scriptsize{\textcircled{\tiny \dag}}}]^{2}Y=0$, then $[(WA)^k]^*(WA)^2[(WA)^{\scriptsize{\textcircled{\tiny \dag}}}]^{2}Y=0$. Thus, $[(WA)^k]^*Y=0$, i.e., $\mathcal{N}(A^{{\scriptsize{\textcircled{\tiny \dag}}}, W})\subseteq \mathcal{N}([(WA)^k]^*)$. Conversely, suppose $Z \in \mathcal{N}([(WA)^k]^*)$, i.e., $[(WA)^k]^*Z=0$, then $A[(WA)^{\scriptsize{\textcircled{\tiny \dag}}}]^2[(WA)^{\scriptsize{\textcircled{\tiny \dag}}}]^{k*}[(WA)^k]^*Z=0$. Therefore, $A^{{\scriptsize{\textcircled{\tiny \dag}}}, W}Z=A[(WA)^{\scriptsize{\textcircled{\tiny \dag}}}]^2Z=0$, i.e., $\mathcal{N}([(WA)^k]^*)\subseteq \mathcal{N}(A^{{\scriptsize{\textcircled{\tiny \dag}}}, W})$. Hence $\mathcal{N}([(WA)^k]^*) = \mathcal{N}(A^{{\scriptsize{\textcircled{\tiny \dag}}}, W})$.
\end{proof}

\begin{prop}~Let $A\in \mathbb{C}^{m\times n},~W\in \mathbb{C}^{n\times m}$ with ${\rm ind}(WA)=k$. Then we have the following facts: \\
$(1)~\mathcal{R}(A^{{\scriptsize{\textcircled{\tiny \dag}}}, W}W)\oplus \mathcal{N}(A^{{\scriptsize{\textcircled{\tiny \dag}}}, W}W)=\mathbb{C}^{m};$\\
$(2)~\mathcal{R}(WA^{{\scriptsize{\textcircled{\tiny \dag}}}, W})\oplus \mathcal{N}(WA^{{\scriptsize{\textcircled{\tiny \dag}}}, W})=\mathbb{C}^{n}.$
\end{prop}

\begin{proof}~(1)~Observe that $A^{{\scriptsize{\textcircled{\tiny \dag}}}, W}W=A[(WA)^{\scriptsize{\textcircled{\tiny \dag}}}]^2W$. For any $X\in \mathbb{C}^{m}$,  $X=A(WA)^{\scriptsize{\textcircled{\tiny \dag}}}WX+[I-A(WA)^{\scriptsize{\textcircled{\tiny \dag}}}W]X$, where   
\begin{equation*}
\begin{aligned}
A(WA)^{\scriptsize{\textcircled{\tiny \dag}}}WX&=A(WA)^{\scriptsize{\textcircled{\tiny \dag}}}WA(WA)^{\scriptsize{\textcircled{\tiny \dag}}}WX
=A(WA)^{\scriptsize{\textcircled{\tiny \dag}}}(WA)^k[(WA)^{\scriptsize{\textcircled{\tiny \dag}}}]^kWX~~~~~~~~\\
&=A[(WA)^{\scriptsize{\textcircled{\tiny \dag}}}]^2(WA)^{k+1}[(WA)^{\scriptsize{\textcircled{\tiny \dag}}}]^kWX=A^{{\scriptsize{\textcircled{\tiny \dag}}}, W}(WA)^{k+1}[(WA)^{\scriptsize{\textcircled{\tiny \dag}}}]^kWX\\
&\in \mathcal{R}(A^{{\scriptsize{\textcircled{\tiny \dag}}}, W}W),
\end{aligned}
\end{equation*}
\begin{equation*}
\begin{aligned}
A^{{\scriptsize{\textcircled{\tiny \dag}}}, W}W[I-A(WA)^{\scriptsize{\textcircled{\tiny \dag}}}W]X
&=A[(WA)^{\scriptsize{\textcircled{\tiny \dag}}}]^2W[I-A(WA)^{\scriptsize{\textcircled{\tiny \dag}}}W]X\\
&=A[(WA)^{\scriptsize{\textcircled{\tiny \dag}}}]^2WX-A[(WA)^{\scriptsize{\textcircled{\tiny \dag}}}]^2WA(WA)^{\scriptsize{\textcircled{\tiny \dag}}}WX\\
&=A[(WA)^{\scriptsize{\textcircled{\tiny \dag}}}]^2WX-A[(WA)^{\scriptsize{\textcircled{\tiny \dag}}}]^2WX=0,~\text{which}~~~~~~~~~~
\end{aligned}
\end{equation*}
implies that
$[I-A(WA)^{\scriptsize{\textcircled{\tiny \dag}}}W]X  \in \mathcal{N}(A^{{\scriptsize{\textcircled{\tiny \dag}}}, W}W).$\\
Therefore, $\mathcal{R}(A^{{\scriptsize{\textcircled{\tiny \dag}}}, W}W)+ \mathcal{N}(A^{{\scriptsize{\textcircled{\tiny \dag}}}, W}W)=\mathbb{C}^{m}$. Further, 
suppose $$Y\in \mathcal{R}(A[(WA)^{\scriptsize{\textcircled{\tiny \dag}}}]^2W) \cap \mathcal{N}(A[(WA)^{\scriptsize{\textcircled{\tiny \dag}}}]^2W),$$ that is to say, $Y=A[(WA)^{\scriptsize{\textcircled{\tiny \dag}}}]^2WZ$ for some $Z\in \mathbb{C}^{m}$ and $A[(WA)^{\scriptsize{\textcircled{\tiny \dag}}}]^2WY=0$. Thus, $A[(WA)^{\scriptsize{\textcircled{\tiny \dag}}}]^2WA[(WA)^{\scriptsize{\textcircled{\tiny \dag}}}]^2WZ=0$, i.e., $A[(WA)^{\scriptsize{\textcircled{\tiny \dag}}}]^3WZ=0$. Pre-multiply this equality by $WAW$, then 
$(WA)^{\scriptsize{\textcircled{\tiny \dag}}}WZ=0$, which deduces that $Y=0$. Hence $\mathcal{R}(A^{{\scriptsize{\textcircled{\tiny \dag}}}, W}W)\oplus \mathcal{N}(A^{{\scriptsize{\textcircled{\tiny \dag}}}, W}W)=\mathbb{C}^{m}$.
\\
(2)~Note that $WA^{{\scriptsize{\textcircled{\tiny \dag}}}, W}=(WA)^{\scriptsize{\textcircled{\tiny \dag}}}$.  From  $\mathcal{R}((WA)^{\scriptsize{\textcircled{\tiny \dag}}})=\mathcal{R}(WA(WA)^{\scriptsize{\textcircled{\tiny \dag}}})$ and $\mathcal{N}((WA)^{\scriptsize{\textcircled{\tiny \dag}}})=\mathcal{N}(WA(WA)^{\scriptsize{\textcircled{\tiny \dag}}})$
as well as $[WA(WA)^{\scriptsize{\textcircled{\tiny \dag}}}]^2=WA(WA)^{\scriptsize{\textcircled{\tiny \dag}}}=[WA(WA)^{\scriptsize{\textcircled{\tiny \dag}}}]^*$, it follows clearly that $\mathcal{R}(WA^{{\scriptsize{\textcircled{\tiny \dag}}}, W})\oplus \mathcal{N}(WA^{{\scriptsize{\textcircled{\tiny \dag}}}, W})=\mathbb{C}^{n}$.
\end{proof}

\begin{prop}~Let $A\in \mathbb{C}^{m\times n},~W\in \mathbb{C}^{n\times m}$ with ${\rm ind}(WA)=k$. Then we have the following facts: \\
$(1)~WAWA^{{\scriptsize{\textcircled{\tiny \dag}}}, W}$ is an orthogonal projector onto $\mathcal{R}((WA)^k);$\\
$(2)~WA^{{\scriptsize{\textcircled{\tiny \dag}}}, W}WA$ is an oblique projector onto $\mathcal{R}((WA)^k)$ along 
$\mathcal{N}([(WA)^k]^{\dag}WA).$
\end{prop}

\begin{proof}~$(1)$~Since $A^{{\scriptsize{\textcircled{\tiny \dag}}}, W}=A[(WA)^{\scriptsize{\textcircled{\tiny \dag}}}]^2$ by applying Theorems \ref{2.1} and \ref{2.2}, then 
$$WAWA^{{\scriptsize{\textcircled{\tiny \dag}}}, W}=WA(WA)^{\scriptsize{\textcircled{\tiny \dag}}}=(WA)^k[(WA)^k]^{\dag}.$$
Therefore, $WAWA^{{\scriptsize{\textcircled{\tiny \dag}}}, W}$ is a orthogonal projector onto $\mathcal{R}((WA)^k)$.\\
$(2)$~Observe that $WA^{{\scriptsize{\textcircled{\tiny \dag}}}, W}WA=(WA)^{\scriptsize{\textcircled{\tiny \dag}}}WA$. Since $(WA)^{\scriptsize{\textcircled{\tiny \dag}}}$ is an outer inverse of $(WA)$, then $[(WA)^{\scriptsize{\textcircled{\tiny \dag}}}WA]^2=(WA)^{\scriptsize{\textcircled{\tiny \dag}}}WA$, together with $$\mathcal{R}((WA)^{\scriptsize{\textcircled{\tiny \dag}}}WA)=\mathcal{R}((WA)^{k})~\text{and}~\mathcal{N}((WA)^{\scriptsize{\textcircled{\tiny \dag}}}WA)=\mathcal{N}([(WA)^{k}]^{\dag}WA),$$ which implies that $WA^{{\scriptsize{\textcircled{\tiny \dag}}}, W}WA$ is a projector onto $\mathcal{R}((WA)^k)$ along $\mathcal{N}([(WA)^k]^{\dag}WA)$.
\end{proof}

\section{Relations among the weighted core-EP inverse and other generalized inverses}
In this section, we wish to reveal the relations among the weighted core-EP inverse, weighted Drazin inverse, the inverse along an element, and $(B,C)$-inverse.

The first  result states that the $W$-weighted core-EP inverse of $A$ (i.e., $A^{{\scriptsize{\textcircled{\tiny \dag}}}, W}$) and  the $W$-weighted Drazin inverse of $A$ (i.e., $A^{D,W}$) can be  mutually expressed
by  post-multiplying an oblique ( orthogonal ) projector.

\begin{thm}~Let $A\in \mathbb{C}^{m\times n},~W\in \mathbb{C}^{n\times m}$ with ${\rm ind}(WA)=k$. Then  \\
$(1)~A^{{\scriptsize{\textcircled{\tiny \dag}}}, W}
=A^{D,W}P_{(WA)^k};$\\
$(2)~A^{D,W}=A^{{\scriptsize{\textcircled{\tiny \dag}}}, W}P_{\mathcal{R}((WA)^k),\mathcal{N}((WA)^k)}.$
\end{thm}

\begin{proof}~$(1)$~It is known that $A^{{\scriptsize{\textcircled{\tiny \dag}}}, W}=A[(WA)^{\scriptsize{\textcircled{\tiny \dag}}}]^2$, $(WA)^{\scriptsize{\textcircled{\tiny \dag}}}=(WA)^D(WA)^k[(WA)^k]^{\dag}$ and $A^{D,W}=A[(WA)^D]^2$. Thus, $A^{{\scriptsize{\textcircled{\tiny \dag}}}, W}=A[(WA)^D]^2(WA)^k[(WA)^k]^{\dag}=A^{D,W}(WA)^k[(WA)^k]^{\dag}=A^{D,W}P_{(WA)^k}$.\\
$(2)$~Observe that $A^{D,W}=A[(WA)^D]^2=A[(WA)^D(WA)^k[(WA)^k]^{\dag}]^2(WA)^k[(WA)^D]^k=A[(WA)^{\scriptsize{\textcircled{\tiny \dag}}}]^2WA(WA)^D=A^{{\scriptsize{\textcircled{\tiny \dag}}}, W}WA(WA)^D=A^{{\scriptsize{\textcircled{\tiny \dag}}}, W}P_{\mathcal{R}((WA)^k),\mathcal{N}((WA)^k)}$.
\end{proof}

In what follows, we investigate the relations between the weighted core-EP inverse and the inverse along an element, $(B,C)$-inverse respectively. Let us recall two  known notions. 
\begin{defn}\emph{\cite{2011}}\label{2011}~Let $A\in \mathbb{C}^{n\times m}$ and $D, X\in \mathbb{C}^{m\times n}$. Then $X$ is the inverse of $A$ along $D$ if
$$XAD=D=DAX~\text{and}~\mathcal{R}_s(X)\subseteq \mathcal{R}_s(D),~\mathcal{R}(X)\subseteq \mathcal{R}(D).$$
\end{defn}

\begin{defn}\emph{\cite{D2012}}\label{2012}~Let $A\in  \mathbb{C}^{n\times m}, B\in  \mathbb{C}^{m\times m}, C\in  \mathbb{C}^{n\times n}, X\in  \mathbb{C}^{m\times n}$. Then $X$ is the $(B,C)$-inverse of $A$ if
$$X\in B\mathbb{C}^{m\times m}X\cap X\mathbb{C}^{n\times n}C~\text{and}~XAB=B,~CAX=C.$$
\end{defn}

\begin{thm}~Let $A\in \mathbb{C}^{m\times n},~W\in \mathbb{C}^{n\times m}$ with ${\rm ind}(WA)=k$. Then  the $W$-weighted core-EP inverse of $A~($ i.e., $A^{{\scriptsize{\textcircled{\tiny \dag}}}, W})$ is the inverse of $WAW$ along $A(WA)^{k}[(WA)^k]^*$.
\end{thm}

\begin{proof}~From Lemma \ref{x} and Theorem \ref{2.2}, it is possible to verify that
\begin{equation*}
\begin{aligned}
A^{{\scriptsize{\textcircled{\tiny \dag}}}, W}WAWA(WA)^{k}[(WA)^k]^*&=A[(WA)^{\scriptsize{\textcircled{\tiny \dag}}}]^2(WA)^{k+2}[(WA)^k]^*~~~~~~~~~~~~~~~~~~~~~~~~~\\
&=A(WA)^{\scriptsize{\textcircled{\tiny \dag}}}(WA)^{k+1}[(WA)^k]^*\\
&=A(WA)^k[(WA)^k]^*,
\end{aligned}
\end{equation*}
\begin{equation*}
\begin{aligned}
A(WA)^k[(WA)^k]^*WAWA^{{\scriptsize{\textcircled{\tiny \dag}}}, W}&=A(WA)^k[(WA)^k]^*WA(WA)^{\scriptsize{\textcircled{\tiny \dag}}}\\
&=A(WA)^k[(WA)^k]^*[WA(WA)^{\scriptsize{\textcircled{\tiny \dag}}}]^*~~~~~~~~~~~~~~~~~~~~~~~\\
&=A(WA)^k[WA(WA)^{\scriptsize{\textcircled{\tiny \dag}}}(WA)^k]^*\\
&=A(WA)^k[(WA)^k]^*,
\end{aligned}
\end{equation*}
\begin{equation*}
\begin{aligned}
A^{{\scriptsize{\textcircled{\tiny \dag}}}, W}&=A[(WA)^{\scriptsize{\textcircled{\tiny \dag}}}]^2
=A(WA)^k[(WA)^k]^{\dag}(WA)^k[(WA)^{\scriptsize{\textcircled{\tiny \dag}}}]^{k+2}\\
&=A(WA)^k[(WA)^k]^*[(WA)^k]^{\dag*}[(WA)^{\scriptsize{\textcircled{\tiny \dag}}}]^{k+2},~\text{i.e.,}~~~~~~~~~~~~~~~~~~~~~~~~~~~~~~~~~~~~~~~\\
\end{aligned}
\end{equation*}
$~~~~~~\mathcal{R}(A^{{\scriptsize{\textcircled{\tiny \dag}}},W})\subseteq \mathcal{R}(A(WA)^k[(WA)^k]^*),$\\
\text{as~well~as},
\begin{equation*}
\begin{aligned}
A^{{\scriptsize{\textcircled{\tiny \dag}}}, W}&=A[(WA)^{\scriptsize{\textcircled{\tiny \dag}}}]^2
=A[(WA)^D]^2(WA)^k[(WA)^k]^{\dag}\\
&=A[(WA)^D]^2[(WA)^k]^{\dag*}[(WA)^k]^*\\
&=A[(WA)^D]^2([(WA)^k]^{\dag}(WA)^k[(WA)^k]^{\dag})^*[(WA)^k]^*~~~~~~~~~~~~~~~~~~~~~~~~~~~~~~~~~\\
&=A[(WA)^D]^2[(WA)^k]^{\dag*}[(WA)^k]^{\dag}(WA)^k[(WA)^k]^*.
\end{aligned}
\end{equation*}
Since $[(WA)^k]^{\dag}(WA)^k=[(WA)^{k+1}]^{\dag}(WA)^{k+1}$(see the dual form of Lemma \ref{x}), then
\begin{equation*}
\begin{aligned}
A^{{\scriptsize{\textcircled{\tiny \dag}}}, W}&=A[(WA)^D]^2[(WA)^k]^{\dag*}[(WA)^{k+1}]^{\dag}(WA)^{k+1}[(WA)^k]^*\\
&=A[(WA)^D]^2[(WA)^k]^{\dag*}[(WA)^{k+1}]^{\dag}WA(WA)^{k}[(WA)^k]^*,~~~~~~~~~~~~~~~~~~~~~~~~~~~~\\
\text{i.e.,}~&\mathcal{R}_s(A^{{\scriptsize{\textcircled{\tiny \dag}}}, W})\subseteq  \mathcal{R}_s(A(WA)^k[(WA)^k]^*).
\end{aligned}
\end{equation*}
Hence $A^{{\scriptsize{\textcircled{\tiny \dag}}}, W}$  is the inverse of $WAW$ along $A(WA)^{k}[(WA)^k]^*$, in view of Definition \ref{2011}.
\end{proof}

\begin{thm}~Let $A\in \mathbb{C}^{m\times n},~W\in \mathbb{C}^{n\times m}$ with $k=\mathrm{max}\{{\rm ind}(AW),~{\rm ind}(WA)\}$. Then the $W$-weighted core-EP inverse of $A~($ i.e., $A^{{\scriptsize{\textcircled{\tiny \dag}}}, W})$ is the $((AW)^k,[(WA)^k]^*)$-inverse of $WAW$.
\end{thm}

\begin{proof}~Clearly, we can verify that 
\begin{equation*}
\begin{aligned}
A^{{\scriptsize{\textcircled{\tiny \dag}}}, W}&=A[(WA)^{\scriptsize{\textcircled{\tiny \dag}}}]^2=A(WA)^k[(WA)^{\scriptsize{\textcircled{\tiny \dag}}}]^{k+2}=(AW)^kA[(WA)^{\scriptsize{\textcircled{\tiny \dag}}}]^{k+2}\\
&=(AW)^kA[(WA)^{\scriptsize{\textcircled{\tiny \dag}}}]^{k+1}WA^{{\scriptsize{\textcircled{\tiny \dag}}}, W}\in (AW)^k\mathbb{C}^{m\times m}A^{{\scriptsize{\textcircled{\tiny \dag}}}, W},
\end{aligned}
\end{equation*}
\begin{equation*}
\begin{aligned}
A^{{\scriptsize{\textcircled{\tiny \dag}}}, W}&=A[(WA)^{\scriptsize{\textcircled{\tiny \dag}}}]^2=A[(WA)^{\scriptsize{\textcircled{\tiny \dag}}}]^2WA(WA)^{\scriptsize{\textcircled{\tiny \dag}}}\\&=A[(WA)^{\scriptsize{\textcircled{\tiny \dag}}}]^2(WA)^k[WA)^{k}]^{\dag}=A[(WA)^{\scriptsize{\textcircled{\tiny \dag}}}]^2[WA)^{k}]^{\dag*}[(WA)^k]^*\\
&=A^{{\scriptsize{\textcircled{\tiny \dag}}}, W}[WA)^{k}]^{\dag*}[(WA)^k]^*
\in A^{{\scriptsize{\textcircled{\tiny \dag}}}, W}\mathbb{C}^{n\times n}[(WA)^k]^*,~\text{as~well~as}, 
\end{aligned}
\end{equation*}
\begin{equation*}
\begin{aligned}
A^{{\scriptsize{\textcircled{\tiny \dag}}}, W}WAW(AW)^k&=A[(WA)^{\scriptsize{\textcircled{\tiny \dag}}}]^2(WA)^{k+1}W=A(WA)^{\scriptsize{\textcircled{\tiny \dag}}}(WA)^kW\\
&=A(WA)^D(WA)^kW=(AW)^D(AW)^{k+1}=(AW)^k,
\end{aligned}
\end{equation*}
\begin{equation*}
\begin{aligned}~
[(WA)^k]^*WAWA^{{\scriptsize{\textcircled{\tiny \dag}}}, W}&=[(WA)^k]^*WA(WA)^{\scriptsize{\textcircled{\tiny \dag}}}
=[(WA)^k]^*[WA(WA)^{\scriptsize{\textcircled{\tiny \dag}}}]^*\\
&=[WA(WA)^{\scriptsize{\textcircled{\tiny \dag}}}(WA)^k]^*=[(WA)^k]^*.
\end{aligned}
\end{equation*}
The above equalities show that $A^{{\scriptsize{\textcircled{\tiny \dag}}}, W}$ is the $((AW)^k,[(WA)^k]^*)$-inverse of $WAW$, in light of Definition \ref{2012}.
\end{proof}

\section{Computational complexities of  representations}
Following  from \cite{S2012,M2017},
the computational complexity of  the pseudoinverse of a singular $m\times n~($resp. $n\times n)$ matrix is denoted by $pinv(m,n)~($resp.
$pinv(n))$; the complexity of multiplying an $m\times n$ matrix by an  $n\times k$ matrix is denoted by $M(m,n,k)$,  abbreviated to  $m\cdot n\cdot k$; the notation $M(n)$ is used instead of $M(n,n,n)$ and is abbreviated to  $n^3$. Let $A\in \mathbb{C}^{m\times n},~W\in \mathbb{C}^{n\times m}$ with $k=$max\{ind$(WA)$,ind$(AW)\}$ and let $l$ be a non-negative integer such that $l\geq k$. In general, an $o(\log l)$ algorithm for matrix exponentiation $A^{l}$ (see \cite{C2001})  would give an algorithm for computing $(AW)^{l}$ in $\mathcal{O}(m^3\log l)$ time, so that $\mathcal{O}((AW)^l)=\mathcal{O}(m^3\log l)$ (see \cite{M2017}). Similarly, $\mathcal{O}((WA)^l)=\mathcal{O}(n^3\log l)$.

\begin{table}[!htp]
\caption{Computational complexity of (2.8)}
\begin{tabular}{ccc}
\hline
Expression~~~~~~~~~~~~~~~~~~~~~~~~~~~~~~~~~~~~&~~~~~~~~~~~~~~~~~~~~~~~~~~~~~~~~~~~~Additional complexity  \\
\hline
$AW$~~~~~~~~~~~~~~~~~~~~~~~~~~~~~~~~~~~~~~~~~~~~&~~~~~~~~~~~~~~~~~~~$m\cdot n\cdot m$  \\
$\Lambda_1=(AW)^l$~~~~~~~~~~~~~~~~~~~~~~~~~~~~~~~~~&~~~~~~~~~~~~~~~~~$m^3\log l$  \\
$\Lambda_2=(AW)^{l+1}=\Lambda_1(AW)$~~~~~~~~~~~~~~~&~~~~~~~~~~~$m^3$\\
$\Lambda_3=W(AW)^{l+1}=W\Lambda_2$~~~~~~~~~~~~~~~~&~~~~~~~~~~~~~~~~~~~$n\cdot m\cdot m$\\
$\Lambda_4=\Lambda_3^{\dag}$~~~~~~~~~~~~~~~~~~~~~~~~~~~~~~~~~~~~~~&~~~~~~~~~~~~~~~~~~~~~$pinv(n,m)$  \\
$X=(AW)^l[W(AW)^{l+1}]^{\dag}=\Lambda_1\Lambda_4$ ~~~~&~~~~~~~~~~~~~~~~~~$m\cdot m\cdot n$\\
\hline
\end{tabular}
\end{table}

\begin{table}[!htp]
\caption{Computational complexity of (2.9)}
\begin{tabular}{ccc}
\hline
Expression~~~~~~~~~~~~~~~~~~~~~~~~~~~~~~~~~~~~&~~~~~~~~~~~~~~~~~~~~~~~~~~~~~~~~~~~~Additional complexity  \\
\hline
$WA$~~~~~~~~~~~~~~~~~~~~~~~~~~~~~~~~~~~~~~~~~~~~&~~~~~~~~~~~~~~~~~~~$n\cdot m\cdot n$  \\
$(WA)^2$~~~~~~~~~~~~~~~~~~~~~~~~~~~~~~~~~~~~~~~~~&~~~~~~~~~~~~$n^3$  \\
$\Lambda_1=(WA)^l$~~~~~~~~~~~~~~~~~~~~~~~~~~~~~~~~~&~~~~~~~~~~~~~~~~~~~~~~~~~$n^3\log (l-1)$  \\
$\Lambda_2=(WA)^{l+2}=\Lambda_1(WA)^2$~~~~~~~~~~~~~~&~~~~~~~~~~~~$n^3$\\
$\Lambda_3=\Lambda_2^{\dag}$~~~~~~~~~~~~~~~~~~~~~~~~~~~~~~~~~~~~~~~&~~~~~~~~~~~~~~~~~~~$pinv(n)$\\
$X=A(WA)^l[(WA)^{l+2}]^{\dag}=A\Lambda_1\Lambda_3$ ~~~~&~~~~~~~~~~~~~~~~~~~~~~$2~m\cdot n\cdot n$\\
\hline
\end{tabular}
\end{table}

The computational complexity of $(2.8)$ can be estimated from the analysis of Table 1:
$$\mathcal{O}(2.8)=3m^2n+m^3+m^3\log l+pinv(n,m).$$
Likewise, the estimation for the computational complexity of $(2.9)$  comes from Table 2:
$$\mathcal{O}(2.9)=3mn^2+2n^3+n^3\log (l-1)+pinv(n).$$
Obviously from $\mathcal{O}(2.8)$ and $\mathcal{O}(2.9)$,  it is more appropriate to use representations involving $AW$  while $m<n$, and use representations involving $WA$  while $m\geq n$. 
In the following, we consider the case: $(0<)m< n$.

\begin{table}[!htp]
\caption{Computational complexity of (1.3)}
\begin{tabular}{ccc}
\hline
Expression~~~~~~~~~~~~~~~~~~~~~~~~~~~~~~~~~~~~&~~~~~~~~~~~~~~~~~~~~~~~~~~~~~~~~~~~~Additional complexity~  \\
\hline
$AW$~~~~~~~~~~~~~~~~~~~~~~~~~~~~~~~~~~~~~~~~~~~~&~~~~~~~~~~~~~~~~~~~$m\cdot n\cdot m$  \\
$\Lambda_1=(AW)^l$~~~~~~~~~~~~~~~~~~~~~~~~~~~~~~~~~&~~~~~~~~~~~~~~~~~$m^3\log l$  \\
$\Lambda_2=(AW)^{l+1}=\Lambda_1(AW)$~~~~~~~~~~~~~~~&~~~~~~~~~~~$m^3$\\
$\Lambda_3=\Lambda_1^{\dag}$~~~~~~~~~~~~~~~~~~~~~~~~~~~~~~~~~~~~~~~&~~~~~~~~~~~~~~~~~~$pinv(m)$\\
$\Lambda_4=W\Lambda_2\Lambda_3$~~~~~~~~~~~~~~~~~~~~~~~~~~~~~~~~&~~~~~~~~~~~~~~~~~~~~~$2~n\cdot m\cdot m$  \\
$X=[W(AW)^{l+1}[(AW)^l]^{\dag}]^{\dag}=\Lambda_4^{\dag}$ ~~~~~&~~~~~~~~~~~~~~~~~~~~~$pinv(n,m)$\\
\hline
\end{tabular}
\end{table}

\begin{table}[!htp]
\caption{ Comparison of representations (1.3) and (2.8). Entries of $A,~W$ are uniformly distributed random numbers from 0 to 1}
\vspace{2mm}
\begin{tabular}{ccccccccc}
\hline
Equation&Size $m,n$&$l\geq k$&CPU Time&$r_1$&$r_2$&$r_3$\\
\hline
(1.3)~~~~~~&                    &          &0.0300 &8.7292e+10 &1.6417e-25&~3.4789e-16 \\
       &100, 200      & $l=k=4$ &&&&\\
(2.8)~~~~~~&                   &         &0.0200& 4.7142e+10 & 8.9982e-26 &~3.7588e-16 \\
\hline
(1.3)~~~~~~&             &              &0.0300& 5.7009e+10&1.0516e-25&~2.1932e-16\\
       &100, 200&$l=k+5$&&&~&~\\
(2.8)~~~~~~&            &               &0.0200& 1.7365e+10&3.2428e-26&~6.6545e-16\\
\hline
(1.3)~~~~~~&~          &~            &0.0400&4.4537e+10&8.1622e-26&~2.1898e-16\\
~~&100, 200~&~$l=k+15$~&~&&~&~\\
(2.8)~~~~~~&          ~&~           &0.0300& 2.5859e+10&4.6722e-26&~2.3790e-16\\
\hline
(1.3)~~~~~~&~          &~            &0.0300&6.1824e+10&1.1411e-25&~2.1261e-16\\
~~&100, 200~&~$l=k+25$~&~&&~&~\\
(2.8)~~~~~~&          ~&~           &0.0300& 1.8199e+10&3.5081e-26&~2.6287e-16\\
\hline
(1.3)~~~~~~&~          &~           &0.2600& 1.2955e+12& 1.4910e-29&~4.6126e-16\\
~~&500, 1000~&~$l=k=3$~&~&&~&~\\
(2.8)~~~~~~&          ~&~           &0.2400&5.5178e+11&1.9805e-30&~5.0956e-16
\\
\hline
(1.3)~~~~~~&~          &~           &0.2600& 1.2955e+12& 1.4910e-29&~4.6126e-16\\
~~&500, 1000~&~$l=k+5$~&~&&~&~\\
(2.8)~~~~~~&          ~&~           &0.2400&5.5178e+11&1.9805e-30&~5.0956e-16
\\
\hline
(1.3)~~~~~~&              ~&          ~&0.3200&1.2142e+12&1.3975e-29&~4.8350e-16\\
        &500, 1000&$l=k+15$&~&&&~\\
(2.8)~~~~~~&               ~&~          &0.2400&5.5196e+11&3.4573e-30&~5.7581e-16\\
\hline
(1.3)~~~~~~&              ~&          ~&0.4700& 7.9785e+11&9.1222e-30&~8.3847e-16\\
        &500, 1000&$l=k+25$&~&&&~\\
(2.8)~~~~~~&               ~&~          &0.2700&5.5190e+11&1.1077e-30&~5.8500e-16\\
\hline
\end{tabular}
\end{table}

The computational complexity of $(1.3)$ is estimated from the analysis of Table 3:
$$\mathcal{O}(1.3)=3m^2n+m^3+m^3\log l+pinv(m)++pinv(n,m).$$
In view of \cite{M2017}  and  \cite{P2009}, the complexity $pinv(m)\geq M(m)=m^3>0$. From $pinv(m)> 0$, it follows that $\mathcal{O}(1.3)>\mathcal{O}(2.8)$. Hence from this perspective, representation (2.8) is better than representation (1.3).
\section{Numerical examples}

Our aim in this section is to  test the time efficiency as well as the accuracy of given representations involving only pseudoinverse, namely, Equalities (1.3) and (2.8). For which,  randomly generated singular matrices  
of different sizes are employed. 
Time efficiency is evaluated by the CPU time and the accuracy is measured by the residual norms.
All the numerical tasks have been performed by using Matlab R2017b.

Let $A\in \mathbb{C}^{m\times n}$ and $W\in \mathbb{C}^{n\times m}$ with ind$(AW)=k$. We assume that $m< n$.
 Approximation derived from a numerical method for computing $A^{{\scriptsize{\textcircled{\tiny \dag}}}, W}$ will be denoted by $X$, and the residual norms in all numerical experiments are denoted by 
$$r_1=||XW(AW)^{k+1}-(AW)^k||_2,~r_2=||AWXWX-X||_2~\text{and}~r_3=||(WAWX)^{*}-WAWX||_2.$$

From  Table 4, the following overall conclusions can be emphasized:

(1) The representation (2.8) gives a better result in the aspect of the computational speed.

(2) Representation (2.8) is  better  in accuracy with respect to the residual norms $r_1$ and  $r_2$.

(3) Contrary to the previous conclusion, the representation (1.3) is  a better expression in accuracy  with respect to norm $r_3$. 

(4) Both (1.3) and (2.8) produce bad results with respect to the norm $r_1$. This reason is the numerical instability caused by various matrix powers.

\section{Conclusion}
This paper introduces several  computational  representations for the $W$-weighted core-EP inverse by using three different matrix decompositions:

$\bullet$ singular-value decomposition;

$\bullet$ full-rank decomposition;

$\bullet$ QR decomposition.
\\
Based on these representations, some properties of the weighted core-EP inverse are derived.  Complexity of introduced representations are estimated  and numerical examples are presented. In addition, the weighted core-EP inverse is considered as a particular $(B, C)$-inverse, and a particular generalized inverse $A^{(2)}_{T,S}$.
\vspace{8mm}

\noindent {\large\bf Acknowledgements}\\
 The authors are highly grateful to the responsible editor and the anonymous referees
 for their valuable and helpful comments and suggestions. The
 first author is grateful to China Scholarship Council for supporting her further study in University of Minho, Portugal.
 \\
 
  \noindent {\large\bf Funding}\\
This research is supported by the National Natural Science Foundation
of China (No.11771076), the Scientific Innovation Research of College Graduates in Jiangsu Province (No.KYZZ16$\_$0112), Partially supported  by   FCT- `Funda\c{c}\~{a}o para a Ci\^{e}ncia e a Tecnologia', within the project UID-MAT-00013/2013.

\end{document}